\documentclass[11pt,oneside,english,reqno]{amsart}
\usepackage[T1]{fontenc}
\usepackage[latin9]{inputenc}
\usepackage{geometry}
\usepackage{babel}
\usepackage{float}
\usepackage{mathrsfs}
\usepackage{amsthm}
\usepackage{amstext}
\usepackage{amssymb}
\usepackage{graphicx}
\usepackage{esint}
\PassOptionsToPackage{normalem}{ulem}
\usepackage{ulem}
\usepackage[unicode=true,pdfusetitle,
 bookmarks=true,bookmarksnumbered=false,bookmarksopen=false,
 breaklinks=false,pdfborder={0 0 1},backref=false,colorlinks=false]
 {hyperref}
\hypersetup{
 colorlinks=true,citecolor=blue,linkcolor=blue,linktocpage=true}

\makeatletter

\providecommand{\tabularnewline}{\\}

\numberwithin{equation}{section}
\numberwithin{figure}{section}
\numberwithin{table}{section}
\usepackage{enumitem}		
\theoremstyle{plain}
\newtheorem{thm}{\protect\theoremname}[section]
  \theoremstyle{remark}
  \newtheorem{rem}[thm]{\protect\remarkname}
  \theoremstyle{definition}
  \newtheorem{example}[thm]{\protect\examplename}
  \theoremstyle{plain}
  \newtheorem{prop}[thm]{\protect\propositionname}
  \theoremstyle{definition}
  \newtheorem{defn}[thm]{\protect\definitionname}
  \theoremstyle{plain}
  \newtheorem{lem}[thm]{\protect\lemmaname}
  \theoremstyle{plain}
  \newtheorem{cor}[thm]{\protect\corollaryname}
  \theoremstyle{remark}
  \newtheorem*{acknowledgement*}{\protect\acknowledgementname}

\allowdisplaybreaks

\providecommand{\MR}[1]{}

\usepackage{mathtools}
\usepackage{refstyle}
\newref{thm}{
         name      = {theorem~},
         names     = {theorems~},
         Name      = {Theorem~},
         Names     = {Theorems~},
         rngtxt    = \RSrngtxt,
         lsttwotxt = \RSlsttxt,
         lsttxt    = \RSlsttxt
         }  
\newref{prop}{
         name      = {proposition~},
         names     = {propositions~},
         Name      = {Proposition~},
         Names     = {Propositions~},
         rngtxt    = \RSrngtxt,
         lsttwotxt = \RSlsttxt,
         lsttxt    = \RSlsttxt
         }   
         
\newref{cor}{
         name      = {corollary~},
         names     = {corollaries~},
         Name      = {Corollary~},
         Names     = {Corollaries~},
         rngtxt    = \RSrngtxt,
         lsttwotxt = \RSlsttxt,
         lsttxt    = \RSlsttxt
         }   
         \newref{rem}{
         name      = {remark~},
         names     = {remarks~},
         Name      = {Remark~},
         Names     = {Remarks~},
         rngtxt    = \RSrngtxt,
         lsttwotxt = \RSlsttxt,
         lsttxt    = \RSlsttxt
         }   
         \newref{ex}{
         name      = {example~},
         names     = {examples~},
         Name      = {Example~},
         Names     = {Examples~},
         rngtxt    = \RSrngtxt,
         lsttwotxt = \RSlsttxt,
         lsttxt    = \RSlsttxt
         }   
         
         \newref{def}{
         name      = {definition~},
         names     = {definitions~},
         Name      = {Definition~},
         Names     = {Definitions~},
         rngtxt    = \RSrngtxt,
         lsttwotxt = \RSlsttxt,
         lsttxt    = \RSlsttxt
         }

\renewcommand{\section}{%
\@startsection{section}{1}%
  \z@{.7\linespacing\@plus\linespacing}{.5\linespacing}%
  {\normalfont\scshape\centering\bfseries}}

\renewcommand{\subsection}{%
\@startsection{subsection}{2}%
  \z@{.5\linespacing\@plus.7\linespacing}{.5\linespacing}%
  {\normalfont\bfseries}}

\renewcommand{\subsubsection}{%
\@startsection{subsubsection}{2}%
  \z@{.5\linespacing\@plus.7\linespacing}{.5\linespacing}%
  {\normalfont\bfseries}}

\newref{sec}{%
      name      = \RSsectxt,
      names     = \RSsecstxt,
      Name      = \RSSectxt,
      Names     = \RSSecstxt,
      refcmd    = {\ifRSstar\S\fi\ref{#1}},
      rngtxt    = \RSrngtxt,
      lsttwotxt = \RSlsttwotxt,
      lsttxt    = \RSlsttxt}

\theoremstyle{definition}

\AtBeginDocument{
  
}

\makeatother

  \providecommand{\acknowledgementname}{Acknowledgement}
  \providecommand{\corollaryname}{Corollary}
  \providecommand{\definitionname}{Definition}
  \providecommand{\examplename}{Example}
  \providecommand{\lemmaname}{Lemma}
  \providecommand{\propositionname}{Proposition}
  \providecommand{\remarkname}{Remark}
\providecommand{\theoremname}{Theorem}

\begin{document}

\title{Unbounded operators, Lie algebras, and local representations}

\author{Palle Jorgensen and Feng Tian}

\address{(Palle E.T. Jorgensen) Department of Mathematics, The University
of Iowa, Iowa City, IA 52242-1419, U.S.A. }

\email{palle-jorgensen@uiowa.edu}

\urladdr{http://www.math.uiowa.edu/\textasciitilde{}jorgen/}

\address{(Feng Tian) Department of Mathematics, Wright State University, Dayton,
OH 45435, U.S.A.}

\email{feng.tian@wright.edu}

\urladdr{http://www.wright.edu/\textasciitilde{}feng.tian/}



\maketitle
\pagestyle{myheadings}
\markright{}
\begin{abstract}
We prove a number of results on integrability and extendability of
Lie algebras of unbounded skew-symmetric operators with common dense
domain in Hilbert space. By integrability for a Lie algebra $\mathfrak{g}$,
we mean that there is an associated unitary representation $\mathcal{U}$
of the corresponding simply connected Lie group such that $\mathfrak{g}$
is the differential of $\mathcal{U}$. Our results extend earlier
integrability results in the literature; and are new even in the case
of a single operator. Our applications include a new invariant for
certain Riemann surfaces.
\end{abstract}
\tableofcontents{}

\section{Introduction}

In this paper we discuss the problem of integrating representations
of Lie algebras to unitary representations of the corresponding simply
connected Lie group. Our main purpose is to stress a link between
the two, taking the form of local representations. Hence we begin
with the case of one dimension, so the real line $\mathbb{R}$, in
section \ref{sec:single}, and we turn to general Lie algebras/Lie
groups in section \ref{sec:general}. The literature is vast, and
to give the reader a sense of different directions, both current and
classical, we suggest the following papers, and the sources cited
there: \cite{Jor87,Jor86,FS66,Seg64,Ben02,Pra91,Voh72,Che72,Arn78,dG84,dG83,Nee11,GKS11,GT12,Nee06,Rob90,Rob89,BGJR88,Rus87,Fro80,Nel59}.
Applications are diverse as well; physics (symmetry groups, relativistic
and non-relativistic), differential equations, harmonic analysis,
and stochastic processes.

The simplest case of the integrability problem for Lie algebras of
unbounded operators is that of a single operator. Here we will restrict
attention to skew-symmetric operators with dense domain in Hilbert
space $\mathscr{H}$. The story begins with von Neumann\textquoteright s
theory of indices, also called defect indices (or deficiency indices);
so named because they measure the gap between an operator being formally
skew-adjoint on the one hand, and skew-adjoint on the other. By the
latter, we mean that it has a spectral resolution, and therefore is
the generator of a strongly continuous one-parameter group of unitary
operators in $\mathscr{H}$. This is really a geometric formulation
of a variety of boundary value problems.

The paper is organized as follows: In section \ref{sec:single} we
present the case of a single operator. This will be used, and it also
allows us to introduce key ideas to be used later, for abelian Lie
algebras in section \ref{sec:abelian}, in the case of non-abelian
Lie algebras of unbounded operators in section \ref{sec:general}.

\section{\label{sec:single}One dimension, single operators}

In the case of a single skew-symmetric operator $T$ with dense domain
in a Hilbert space $\mathscr{H}$, we introduce a notion of \textquotedblleft local
invariance,\textquotedblright{} and we prove that it is equivalent
to $T$ having von Neumann indices $(0,0)$, i.e., to $T$ being essentially
skew-adjoint. Equivalently, indices $(0,0)$ means that $T$ has a
projection valued spectral resolution; and is therefore the infinitesimal
generator of a strongly continuous one-parameter group of unitary
operators in $\mathscr{H}$.

The systematic study of extensions of symmetric (or equivalently skew-symmetric)
operators began with von Neumann\textquoteright s paper \cite{vN30}.
Different applications, including boundary value problems and scattering
theory, see \cite{MR1009163} and \cite{MR1037774}. For more recent
applications, see for example \cite{MR3167762,MR3129890,MR3017129,MR2989487}.
\begin{thm}
\label{thm:flow1}Let $H$ be a skew-symmetric operator with dense
domain $\mathscr{D}$ in a Hilbert space $\mathscr{H}$, i.e., 
\begin{equation}
\left\langle Hv,w\right\rangle +\left\langle v,Hw\right\rangle =0\label{eq:i1}
\end{equation}
for all $v,w\in\mathscr{D}$. Suppose there are subspaces $\mathscr{D}_{\varepsilon}$,
$\varepsilon\in\mathbb{R}_{+}$, such that
\begin{enumerate}[label=(\roman{enumi}),ref=(\roman{enumi})]
\item \label{enu:i1-1}
\[
\mathscr{D}=\bigcup_{\varepsilon\in\mathbb{R}_{+}}\mathscr{D}_{\varepsilon}.\;\mbox{(Note that \ensuremath{\mathscr{D}_{\varepsilon}} may be zero if \ensuremath{\varepsilon}\ is large.)}
\]

\item \label{enu:i1-2}For every $\varepsilon\in\mathbb{R}_{+}$ there is
an $\varepsilon'$, $0<\varepsilon'<\varepsilon$, and there are operators
\[
\left\{ \varphi_{\varepsilon}\left(t\right):\left|t\right|<\varepsilon\right\} 
\]
 with dense domain $\mathscr{D}$ such that:

\begin{enumerate}
\item \label{enu:i1-2-a}$\varphi_{\varepsilon}\left(s+t\right)=\varphi_{\varepsilon}\left(s\right)\varphi_{\varepsilon}\left(t\right)$,
$\left|s\right|<\varepsilon$, $\left|t\right|<\varepsilon$, $\left|s+t\right|<\varepsilon$; 
\item \label{enu:i1-2-b}$\frac{d}{dt}\varphi_{\varepsilon}\left(t\right)=H\varphi_{\varepsilon}\left(t\right)$,
$\left|t\right|<\varepsilon'$, $\varphi_{\varepsilon}\left(0\right)=I$;
and 
\item \label{enu:i1-2-c}$\varphi_{\varepsilon}\left(t\right)$ leaves $\mathscr{D}_{\varepsilon}$
invariant for all $t\in\left(-\varepsilon',\varepsilon'\right)$. 
\end{enumerate}
\end{enumerate}
\end{thm}
Then the operator $H$ is essentially skew-adjoint, i.e., it has a
projection-valued spectral resolution. 
\begin{proof}
\textbf{\uline{Step 1}} For all $\varepsilon\in\mathbb{R}_{+}$,
and $v\in\mathscr{D}_{\varepsilon}$, we have 
\begin{equation}
\left\Vert \varphi_{\varepsilon}\left(t\right)v\right\Vert =\left\Vert v\right\Vert ,\;\forall t\in\left(-\varepsilon,\varepsilon\right).\label{eq:i2}
\end{equation}
Indeed, let $v\in\mathscr{D}_{\varepsilon}$ be as above; then 
\begin{eqnarray*}
\frac{d}{dt}\left\Vert \varphi_{\varepsilon}\left(t\right)v\right\Vert ^{2} & = & \left\langle \frac{d}{dt}\varphi_{\varepsilon}\left(t\right)v,\varphi_{\varepsilon}\left(t\right)v\right\rangle +\left\langle \varphi_{\varepsilon}\left(t\right)v,\frac{d}{dt}\varphi_{\varepsilon}\left(t\right)v\right\rangle \\
 & \underset{\left(\text{by }\ref{enu:i1-2-b}\right)}{=} & \left\langle H\varphi_{\varepsilon}\left(t\right)v,\varphi_{\varepsilon}\left(t\right)v\right\rangle +\left\langle \varphi_{\varepsilon}\left(t\right)v,H\varphi_{\varepsilon}\left(t\right)v\right\rangle \\
 & \underset{\left(\text{by }\left(\ref{eq:i1}\right)\right)}{=} & 0.
\end{eqnarray*}
Hence (\ref{eq:i2}) follows. 

\textbf{\uline{Step 2}} We get a strongly continuous unitary one-parameter
group $\left\{ U_{t}\:|\: t\in\mathbb{R}\right\} $ acting in $\mathscr{H}$,
by extending from the local flows $\left\{ \varphi_{\varepsilon}\left(t\right)\:|\:\left|t\right|<\varepsilon\right\} $,
as follows: 

Pick $t\in\mathbb{R}\backslash\left\{ 0\right\} $, and $n\in\mathbb{N}$
s.t. $\left|t/n\right|<\varepsilon'$, where $\varepsilon'$ is as
in \ref{enu:i1-2}. Then the operators $\varphi_{\varepsilon}\left(t/n\right)$
leave invariant the fixed subspaces $\mathscr{D}_{\varepsilon}$ from
\ref{enu:i1-1}. For $v\in\mathscr{D}_{\varepsilon}$, we then get,
using \ref{enu:i1-2-a}, 
\begin{equation}
U_{t}v:=\left(\varphi_{\varepsilon}\left(\frac{t}{n}\right)\right)^{n}v.\label{eq:il3}
\end{equation}
An application of (\ref{eq:i2}) yields, 
\begin{equation}
\left\Vert U_{t}v\right\Vert =\left\Vert v\right\Vert .\label{eq:il4}
\end{equation}
Hence, $U_{t}$ extends by closure (since \ref{enu:i1-1} holds and
$\mathscr{D}$ is dense) to a unitary operator $U_{t}:\mathscr{H}\rightarrow\mathscr{H}$
for all $t\in\mathbb{R}$. By (\ref{eq:i2}), the set of pairs in
$\mathbb{R}$ for which the group 
\begin{equation}
U_{s+t}=U_{s}U_{t}\label{eq:il5}
\end{equation}
holds is open, closed, and non-empty. Using \ref{enu:i1-2-a} and
connectedness of $\mathbb{R}$, we conclude that (\ref{eq:il5}) holds
for all $s,t\in\mathbb{R}$. 

\textbf{\uline{Step 3}} $t\longrightarrow U_{t}$ is strongly continuous:
This follows from \ref{enu:i1-2-b}.

\textbf{\uline{Step 4}} By Stone's theorem, $\left\{ U_{t}\right\} _{t\in\mathbb{R}}$
has a unique infinitesimal generator, i.e., there is a selfadjoint
operator $\widetilde{H}$ (generally unbounded) such that
\begin{equation}
U_{t}=e^{it\widetilde{H}}=\int_{\mathbb{R}}e^{i\lambda t}P\left(d\lambda\right)\label{eq:il6}
\end{equation}
where $P\left(\cdot\right)$ denotes the spectral resolution of the
selfadjoint operator $\widetilde{H}$. 

\textbf{\uline{Step 5}} The operator $i\widetilde{H}$ extends
$H$, i.e., $\mathscr{D}$ is contained in $dom(\widetilde{H})$,
and 
\begin{equation}
Hv=i\widetilde{H}v=\lim_{t\rightarrow0}\frac{U_{t}v-v}{t}\label{eq:il7}
\end{equation}
holds for all $v\in\mathscr{D}.$

Proof of (\ref{eq:il7}). Pick $v\in\mathscr{D}$, $\varepsilon>0$
s.t. $v\in\mathscr{D}_{\varepsilon}$, see \ref{enu:i1-1}; then for
$\left|t\right|<\varepsilon$, $\varphi_{\varepsilon}\left(t\right)v=U_{t}v$,
and 
\[
\frac{d}{dt}\big|_{t=0}\varphi_{\varepsilon}\left(t\right)\underset{(\text{by (ii)(b)})}{=}Hv=\lim_{t\rightarrow0}\frac{U_{t}v-v}{t}
\]
exists, and by definition of the infinitesimal generator of $\left\{ U_{t}\right\} _{t\in\mathbb{R}}$,
we get $Hv=i\widetilde{H}v$ which is the desired conclusion (\ref{eq:il7}). 

We now show that the closure of $\left(H,\mathscr{D}\right)$, the
operator $H$ with $\mathscr{D}$ as its domain, is $i\widetilde{H}$,
i.e., that infinitesimal generator of the unitary one-parameter group
$\left\{ U_{t}\right\} _{t\in\mathbb{R}}$ from (\ref{eq:il6})-(\ref{eq:il7})
is the closure of $H$.

\uline{Details:} Since we already proved that $i\widetilde{H}$
is an extension of $H$ on $\mathscr{D}$, i.e., that the closure
of the graph 
\[
G\left(H\right):=\left\{ \begin{pmatrix}v\\
Hv
\end{pmatrix},v\in\mathscr{D}\right\} 
\]
in $\mathscr{H}\times\mathscr{H}$ is contained in the graph of $i\widetilde{H}$,
and $G\left(H\right)^{cl}=G(i\widetilde{H})$; it is enough to prove
that $\left(H,\mathscr{D}\right)$ has deficiency indices $\left(0,0\right)$;
for then the skew-adjoint extension $i\widetilde{H}$ is unique, so
it must be the closure of $H$.

To verify the index assertion, we must show that the following two
equations,
\begin{equation}
\left\langle Hv\pm v,f_{\pm}\right\rangle \equiv0,\;\forall v\in\mathscr{D}\label{eq:il8}
\end{equation}
has only solution $f_{\pm}=0$ in $\mathscr{H}$. By symmetry, we
need only consider one of the equations in (\ref{eq:il8}), say the
$f_{-}$ case: 

For $v\in\mathscr{D}=\bigcup_{\varepsilon\in\mathbb{R}_{+}}\mathscr{D}_{\varepsilon}$,
we pick $\varepsilon$ and $\varepsilon'$ as in \ref{enu:i1-2-c}
in the statement of the theorem. Using formula (\ref{eq:il3}), we
conclude that $U_{t}v\in\mathscr{D}$ for all $t\in\mathbb{R}$. Hence,
by (\ref{eq:il8}), we get 
\[
\left\langle HU_{t}v,f_{-}\right\rangle =\left\langle U_{t}v,f_{-}\right\rangle ,\;\forall t\in\mathbb{R}.
\]
But by (\ref{eq:il7}) this is equivalent to the following differential
equation:
\[
\frac{d}{dt}\left\langle U_{t}v,f_{-}\right\rangle =\left\langle U_{t}v,f_{-}\right\rangle ,\;\forall t\in\mathbb{R},
\]
with initial condition $\left\langle U_{t}v,f_{-}\right\rangle \Big|_{t=0}=\left\langle v,f_{-}\right\rangle $.
Hence 
\begin{equation}
\left\langle U_{t}v,f_{-}\right\rangle =e^{t}\left\langle v,f_{-}\right\rangle ,\;\forall t\in\mathbb{R}.\label{eq:il9}
\end{equation}
But since $U_{t}$ is unitary (isometric) the LHS in (\ref{eq:il9})
is bounded as a function of $t\in\mathbb{R}$, while the RHS in (\ref{eq:il9})
is always unbounded when $\left\langle v,f_{-}\right\rangle \neq0$.
We conclude therefore that $\left\langle v,f_{-}\right\rangle =0$
for all $v\in\mathscr{D}$. But, by condition \ref{enu:i1-1} in the
theorem, $\mathscr{D}$ is dense in $\mathscr{H}$, and so $f_{-}=0$. \end{proof}
\begin{rem}
The converse to the implication in Theorem \ref{thm:flow1} holds;
in fact a slightly stronger version holds. I.e., it holds that every
skew-adjoint operator $i\widetilde{H}$ $(\widetilde{H}^{*}=\widetilde{H})$
in a Hilbert space $\mathscr{H}$ admits dense subspaces $\mathscr{D}$
contained in $dom(i\widetilde{H})$ which satisfy the conditions \ref{enu:i1-1}-\ref{enu:i1-2}
from the statement of Theorem \ref{thm:flow1}.

In fact, there are many such choices for $\mathscr{D}$; and it becomes
more of a question of identifying choices that are useful in applications.
Below we sketch a choice of dense subspace $\mathscr{D}$ (subject
to \ref{enu:i1-1}-\ref{enu:i1-2}) for a given skew-adjoint operator
$i\widetilde{H}$ in $\mathscr{H}$, where $\mathscr{H}$ is a fixed
Hilbert space.

Given $i\widetilde{H}$, we get a projection valued measure $P\left(\cdot\right)$
as in (\ref{eq:il6}); i.e., $P\left(\cdot\right)$ is a sigma-additive
function defined on $\mathscr{B}\left(\mathbb{R}\right)$, Borel sets,
such that 
\begin{gather*}
P\left(\mathbb{R}\right)=I_{\mathscr{H}}\\
P\left(A\right)=P\left(A\right)^{*}=P\left(A\right)^{2},\;\forall A\in\mathscr{B}\left(\mathbb{R}\right)\\
P\left(A\cap B\right)=P\left(A\right)P\left(B\right),\;\forall A,B\in\mathscr{B}\left(\mathbb{R}\right)
\end{gather*}
and s.t. (\ref{eq:il6}) holds. 

We then set, for all $\varepsilon\in\mathbb{R}_{+}$, $\mathscr{D}_{\varepsilon}:=P\left(-\varepsilon^{-1},\varepsilon^{-1}\right)\mathscr{H}\subset\mathscr{H}$,
and $\mathscr{D}:=\bigcup_{\varepsilon\in\mathbb{R}_{+}}\mathscr{D}_{\varepsilon}$.
It then follows from basic spectral theory that this $\mathscr{D}$
satisfies the conditions from \ref{enu:i1-1}-\ref{enu:i1-2} in the
statement of Theorem \ref{thm:flow1}.\end{rem}
\begin{example}
Let $\mathscr{H}=L^{2}(0,1)$, and $\mathscr{D}=C_{c}^{1}\left(0,1\right)=$
compactly supported $C^{1}$ functions in $\left(0,1\right)$. And
$H=d/dx$, skew-symmetric on $\mathscr{D}$. If $v$ is in $\mathscr{D}$
we can exponentiate locally as 
\[
\varphi\left(t\right)v\left(\cdot\right):=v\left(\cdot-t\right)
\]
as long as we don't translate out of $(0,1)$, but $H$ is not essentially
skew-adjoint; it has deficiency indices $(1,1)$. Note that $\varphi\left(t\right)$
does not satisfy the local invariance condition. Reason: The support
of $v\left(\cdot-t\right)$ gets closer to a boundary point as $\left|t\right|\neq0$.

To see that $H$ above with $\mathscr{D}=dom\left(H\right)=C_{c}^{\infty}\left(0,1\right)$
has von Neumann indices $\left(1,1\right)$, note that each of the
two equations 
\[
H^{*}f_{\pm}=\mp f_{\pm}
\]
has non-zero solution, i.e., $f_{\pm}\in L^{2}\left(0,1\right)$,
$f_{\pm}\in dom\left(H^{*}\right)$. The solutions are 
\[
f_{\pm}\left(x\right)=\mbox{const}\cdot e^{\pm x},\; x\in\left(0,1\right).
\]

\end{example}

\section{\label{sec:abelian}Commuting skew-symmetric operators with common
dense domain}

Here we study the case of integrability and extendability of abelian
Lie algebras of unbounded skew-symmetric operators with common dense
domain in Hilbert space.
\begin{thm}
\label{thm:flow2}Fix $n\in\mathbb{N}$, and set $J_{n}:=\left\{ 1,\ldots,n\right\} $.
Let $H_{i}$, $i\in J$, be a set of skew-symmetric operators with
a common dense domain $\mathscr{D}$ in a Hilbert space $\mathscr{H}$,
i.e., 
\begin{equation}
\left\langle H_{j}v,w\right\rangle +\left\langle v,H_{j}w\right\rangle =0\label{eq:m1}
\end{equation}
for all $v,w\in\mathscr{D}$, $j\in J_{n}$. Suppose there are subspaces
$\mathscr{D}_{\varepsilon}$, $\varepsilon\in\mathbb{R}_{+}$, such
that
\begin{enumerate}[label=(\roman{enumi}'),ref=(\roman{enumi}')]
\item \label{enu:m-1}
\[
\mathscr{D}=\bigcup_{\varepsilon\in\mathbb{R}_{+}}\mathscr{D}_{\varepsilon}.
\]

\item \label{enu:m-2}For every $\varepsilon\in\mathbb{R}_{+}$ there is
an $\varepsilon'$, $0<\varepsilon'<\varepsilon$, and there are operators
\[
\left\{ \varphi_{\varepsilon,j}\left(t\right):\left|t\right|<\varepsilon\right\} 
\]
 with dense domain $\mathscr{D}$ such that: \\
For all $j\in J_{n}$, 

\begin{enumerate}
\item \label{enu:m-1a}$\varphi_{\varepsilon,j}\left(s+t\right)=\varphi_{\varepsilon,j}\left(s\right)\varphi_{\varepsilon,j}\left(t\right)$,
$\left|s\right|<\varepsilon$, $\left|t\right|<\varepsilon$, $\left|s+t\right|<\varepsilon$; 
\item \label{enu:m-1b}$\frac{d}{dt}\varphi_{\varepsilon,j}\left(t\right)=H_{j}\varphi_{\varepsilon,j}\left(t\right)$,
$\left|t\right|<\varepsilon'$, $\varphi_{\varepsilon,j}\left(0\right)=I$;
\item \label{enu:m-1c}$\varphi_{\varepsilon,j}\left(t\right)$ leaves $\mathscr{D}_{\varepsilon}$
invariant for all $t\in\left(-\varepsilon',\varepsilon'\right)$;
and 
\item \label{enu:m-1d} 
\[
\varphi_{\varepsilon,j}\left(s\right)\varphi_{\varepsilon,j'}\left(t\right)=\varphi_{\varepsilon,j'}\left(t\right)\varphi_{\varepsilon,j}\left(s\right)
\]
for all $j,j'\in J_{n}$, $\left|s\right|,\left|t\right|<\varepsilon$.
\end{enumerate}
\end{enumerate}
\end{thm}
Then the operators $H_{j}$ are essentially skew-adjoint, and the
operator closures $\overline{H_{j}}$ strongly commute, i.e., the
operators $\overline{H_{j}}$ have commuting spectral projections.
\begin{proof}
Note steps 1-5 in the proof of Theorem \ref{thm:flow1} carry over
for each $H_{j}$, $j\in J_{n}$. 

Indeed, for each $j\in J_{n}$, define $U_{j}\left(t\right)$ as in
(\ref{eq:il3}), acting on $\mathscr{D}$. Then $U_{j}\left(t\right)$
is norm-preserving (see (\ref{eq:il4})), and by \ref{enu:m-1d},
we have
\[
U_{j}\left(t\right)U_{j'}\left(s\right)=U_{j'}\left(s\right)U_{j}\left(t\right)\;\mbox{on }\mathscr{D},\;\forall j,j'\in J_{n},\: t,s\in\mathbb{R}.
\]
Therefore, the operators $U_{j}$ extend by continuity to a family
of commuting unitary one-parameter groups in $\mathscr{H}$. By Stone's
theorem, 
\[
U_{j}\left(t\right)=e^{it\widetilde{H_{j}}}=\int e^{it\lambda}P_{j}\left(d\lambda\right),\; t\in\mathbb{R}
\]
where $\widetilde{H_{j}}$ is the corresponding selfadjoint infinitesimal
generator, and $P_{j}$ the projection valued measure. Moreover, $H_{j}\subset i\widetilde{H_{j}}$
and $\overline{H_{j}}=i\widetilde{H_{j}}$, i.e., $H_{j}$ is essentially
selfadjoint. 

For finish the proof, we recall a general theorem in the theory of
integrable representations of $*$-algebras. See \cite[Lemma 1]{Jor76}
and a complete proof in \cite{JM84}. 

Details: Note \ref{enu:m-1}-\ref{enu:m-2} define a representation
$\rho$ of an $n$-dimensional abelian Lie algebra $\mathfrak{g}\left(=\mathbb{R}^{n}\right)$
acting on the common dense domain $\mathscr{D}$ in $\mathscr{H}$.
By \cite[Lemma 1]{Jor76}, the local invariance condition and the
density of $\mathscr{D}$ imply that $\rho$ can be exponentiated
(i.e., $\rho$ is integrable) to a unitary representation $\mathcal{U}$
of the Lie group $G=\left(\mathbb{R}^{n},+\right)$, and $\rho=d\mathcal{U}$.
Therefore, the generators $i\widetilde{H_{j}}$ strongly commute.
\end{proof}
Combining the ideas above, we get the following result for commuting
operators. We state it for $n=2$, but the conclusions hold \emph{mutatis
mutandis} for the case $n>2$ as well.
\begin{prop}
Let $H_{1}$ and $H_{2}$ be two skew-symmetric operators defined
on a common dense domain $\mathscr{D}$ in a Hilbert space $\mathscr{H}$.
Assume that $H_{j}\mathscr{D}\subseteq\mathscr{D}$, $j=1,2$. Then
the following conditions are equivalent:
\begin{enumerate}
\item the operator $L:=H_{1}^{2}+H_{2}^{2}$ is essentially selfadjoint
on $\mathscr{D}$; 
\item each operator $H_{j}$ is essentially skew-adjoint and the two unitary
one-parameter groups $U_{j}\left(t\right):=e^{t\overline{H_{j}}}$,
$j=1,2$ are commuting;
\item each operator $H_{j}$ is essentially skew-adjoint and 
\[
U\left(t_{1},t_{2}\right):=e^{t_{1}\overline{H_{1}}}e^{t_{2}\overline{H_{2}}},\;\left(t_{1},t_{2}\right)\in\mathbb{R}^{2}
\]
 defines a strongly continuous unitary representation of $\left(\mathbb{R}^{2},+\right)$
acting on $\mathscr{H}$; 
\item for each $j$ and $\lambda\in\mathbb{C}\backslash i\mathbb{R}$, the
operator ranges $\left(\lambda-H_{j}\right)\mathscr{D}$ are dense
in $\mathscr{H}$, and the bounded operators 
\[
\left(\lambda_{1}-H_{1}\right)^{-1}\;\mbox{and }\;\left(\lambda_{2}-H_{2}\right)^{-1}
\]
are commuting, $\forall\lambda_{j}\in\mathbb{C}\backslash i\mathbb{R}$;
and
\item the conditions \ref{enu:m-1} and \ref{enu:m-2} in Theorem \ref{thm:flow2}
hold.
\end{enumerate}
\end{prop}

\section{\label{sec:general}Lie algebras, and local representations of Lie
groups}

We begin with rigorous definitions of the following two notions: integrability
and extendability for Lie algebras $\mathfrak{g}$ of unbounded skew-symmetric
operators with common dense domain in Hilbert space. In our main result,
Corollary \ref{cor:inv}, we show that a given finite-dimensional
Lie algebra $\mathfrak{g}$ of skew-symmetric operators is integrable
to a unitary representation of the corresponding simply connected
Lie group if and only if it has a dense and locally invariant domain.
\begin{defn}
Let $\mathscr{H}$ be a Hilbert space, and $\mathscr{D}$ a dense
subspace. Let $\mathfrak{g}$ be a finite dimensional Lie algebra
over $\mathbb{R}$. Let $Sk\left(\mathscr{D}\right)$ denote the real
Lie algebra of all linear operators $X$ satisfying
\begin{enumerate}[label=(\roman{enumi})]
\item $\mathscr{D}\subset dom\left(X\right)$, for all $X\in\mathfrak{g}$;
\item $X\left(\mathscr{D}\right)\subset\mathscr{D}$;
\item $\left\langle Xu,w\right\rangle +\left\langle u,Xw\right\rangle =0$,
for all $X\in\mathfrak{g}$, and $u,w\in\mathscr{D}$. 
\end{enumerate}

A representation $\rho$ of $\mathfrak{g}$ is a Lie-homomorphism
$\rho:\mathfrak{g}\rightarrow Sk\left(\mathscr{D}\right)$, i.e.,
\begin{equation}
\rho\left(\left[x,y\right]\right)=\left[\rho\left(x\right),\rho\left(y\right)\right],\;\forall x,y\in\mathfrak{g}.\label{eq:r1}
\end{equation}
Occasionally, we shall use the notation $X=\rho\left(x\right)$, $x\in\mathfrak{g}$.

\end{defn}

\begin{defn}
\label{def:ext}We say that a representation $\left(\rho,\mathfrak{g},\mathscr{D}\right)$
has an integrable extension iff (Def) there is a unitary representation
$\mathcal{U}$ of the simply connected Lie group $G$ with $\mathfrak{g}$
as its Lie algebra, s.t. 
\begin{equation}
\rho\left(x\right)\subseteq d\mathcal{U}\left(x\right),\;\forall x\in\mathfrak{g},\label{eq:r2}
\end{equation}
where the containment ``$\subseteq$'' in (\ref{eq:r2}) refers
to containment of graphs, i.e.,
\begin{equation}
\begin{split}\mathscr{D}\subseteq dom\left(d\mathcal{U}\left(x\right)\right),\;\forall x\in\mathfrak{g},\:\mbox{and}\\
\rho\left(x\right)w=d\mathcal{U}\left(x\right)w,\;\forall w\in\mathscr{D}.
\end{split}
\label{eq:r3}
\end{equation}

\end{defn}

\begin{defn}
\label{def:int}We say that a Lie algebra representation $\left(\rho,\mathfrak{g},\mathscr{D}\right)$
is integrable if (\ref{eq:r2}) holds, but with equality for the closure,
i.e.,
\begin{equation}
\begin{split}\mbox{Graph}\left(\rho\left(x\right)\right)^{\mathscr{H}\times\mathscr{H}\:\mbox{closure}}=\mbox{Graph}\left(d\mathcal{U}\left(x\right)\right),\;\mbox{where}\\
d\mathcal{U}\left(x\right)w=\lim_{t\rightarrow0}\frac{\mathcal{U}\left(\exp tx\right)w-w}{t}
\end{split}
\label{eq:r4}
\end{equation}
 \end{defn}
\begin{lem}[\cite{JM84}]
Let $G$ be a Lie group with Lie algebra $\mathfrak{g}$, and exponential
mapping $\mathfrak{g}\xrightarrow{\exp}G$, and let $\mathcal{U}$
be a unitary representation of $G$ acting on a Hilbert space $\mathscr{H}$.
Set 
\begin{equation}
\mathscr{H}_{\infty}:=\left\{ w\in\mathscr{H}\:\big|\:\left(G\ni g\rightarrow\mathcal{U}\left(g\right)w\right)\in C^{\infty}\left(G,\mathscr{H}\right)\right\} ,\label{eq:r5}
\end{equation}
the $C^{\infty}$-vectors of $\mathcal{U}$; then
\begin{equation}
d\mathcal{U}\left(x\right)w=\lim_{t\rightarrow0}\frac{\mathcal{U}\left(\exp tx\right)w-w}{t}\label{eq:r6}
\end{equation}
is well defined for all $w\in\mathscr{H}_{\infty}$, $x\in\mathfrak{g}$. 

Moreover, $d\mathcal{U}\left(x\right)$ on $\mathscr{H}_{\infty}$
is essentially skew-adjoint; i.e., 
\[
\left(d\mathcal{U}\left(x\right)\big|{}_{\mathscr{H}_{\infty}}\right)^{*}=-d\mathcal{U}\left(x\right),\;\forall x\in\mathfrak{g}.
\]

\end{lem}
\uline{Note.} If $X$ is an operator with dense domain, then domain
of its adjoint $X^{*}$ is $\left\{ w\in\mathscr{H}\:\big|\:\exists C=C_{w}<\infty\;\mbox{s.t. }\left|\left\langle w,Xu\right\rangle \right|\leq C\left\Vert u\right\Vert ,\;\forall u\in dom\left(X\right)\right\} $.
\begin{example}
\label{ex:1D}Let $H_{1}=\frac{d}{dx}\big|_{C_{c}^{\infty}\left(0,1\right)}$
in $L^{2}\left(0,1\right)$. $H$ is densely defined, skew-symmetric,
with deficiency indices $\left(1,1\right)$. $H_{1}$ is extendable
but not integrable. 

On the other hand, $H_{2}=\frac{d}{dx}\big|_{C_{c}^{\infty}\left(\mathbb{R}\right)}$
is densely defined, skew-symmetric, acting in $L^{2}\left(\mathbb{R}\right)$,
and it has deficiency indices $\left(0,0\right)$; i.e., $\overline{H_{2}}^{*}=-\overline{H_{2}}$,
skew-adjoint. $\overline{H_{2}}$ generates the one-parameter unitary
group $\left\{ U\left(t\right)=e^{-t\overline{H_{2}}}\right\} _{t\in\mathbb{R}}$,
where 
\[
U\left(t\right)f\left(x\right)=f\left(x-t\right)
\]
for all $f\in L^{2}\left(\mathbb{R}\right)$. Therefore, $H_{2}$
is integrable.
\end{example}

\begin{example}
\label{ex:log}Let $M$ denote the Riemann surface of the complex
$\log z$ function. We will realize $M$ as a covering space for $\mathbb{R}^{2}\backslash\left\{ \left(0,0\right)\right\} $
with an infinite number of sheets indexed by $\mathbb{Z}$ as follows
(see Fig \ref{fig:logz}):

\begin{figure}[H]
\begin{tabular}{cc}
\includegraphics[scale=0.6]{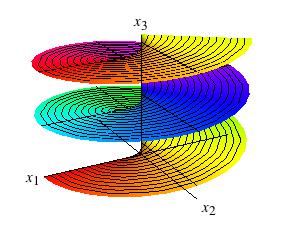} & \includegraphics[scale=0.6]{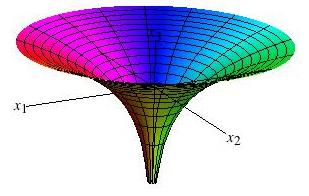}\tabularnewline
\end{tabular}

\protect\caption{\label{fig:logz}$M$ the Riemann surface of $\log z$ as an $\infty$
cover of $\mathbb{R}^{2}\backslash\left\{ \left(0,0\right)\right\} $. }

\end{figure}

Rotate the real $\log x$, $x\in\mathbb{R}_{+}$, in the $\left(x_{1},x_{2}\right)$
plane while creating spirals in the $x_{3}$-direction, one full rotation
for each interval $n\leq x_{3}<n+1$, $n\in\mathbb{Z}$. 

The measure of $L^{2}\left(M\right)$ and $C_{c}^{\infty}\left(M\right)$
derive from the i.e., the 2D-Lebesgue measure lifts to a unique measure
on $M$; hence $L^{2}\left(M\right)$. Here the two skew symmetric
operators $\frac{\partial}{\partial x_{j}}$, $j=1,2$ with domain
$C_{c}^{\infty}\left(M\right)$ define an abelian 2-dimensional Lie
algebra of densely defined operators in the Hilbert space $L^{2}\left(M\right)$. \end{example}
\begin{prop}
\label{prop:logz}(i) The $\left\{ \frac{\partial}{\partial x_{j}}\right\} _{j=1,2}$
Lie algebra with domain $C_{c}^{\infty}\left(M\right)\subset L^{2}\left(M\right)$
is not extendable (see Definition \ref{def:ext}).

(ii) Each operator $\frac{\partial}{\partial x_{j}}$ on $C_{c}^{\infty}\left(M\right)$
is essentially skew-adjoint, i.e., 
\begin{equation}
-\left(\frac{\partial}{\partial x_{j}}\Big|_{C_{c}^{\infty}\left(M\right)}\right)^{*}=\mbox{closure}\left(\frac{\partial}{\partial x_{j}}\Big|_{C_{c}^{\infty}\left(M\right)}\right),\; j=1,2.\label{eq:a1}
\end{equation}

(iii) The two skew-adjoint operators in (\ref{eq:a1}) are \uline{not}
strongly commuting.

(iv) The operator 
\begin{equation}
L:=\left(\frac{\partial}{\partial x_{1}}\right)^{2}+\left(\frac{\partial}{\partial x_{2}}\right)^{2}\;\mbox{on }C_{c}^{\infty}\left(M\right)\label{eq:a2}
\end{equation}
has deficiency indices $\left(\infty,\infty\right)$. \end{prop}
\begin{proof}
The two operators $\frac{\partial}{\partial x_{j}}$ generate unitary
one-parameter groups $U_{j}\left(t\right)$, $j=1,2$, acting on $L^{2}\left(M\right)$
since the two coordinate translations
\begin{eqnarray}
\left(x_{1},x_{2}\right) & \longmapsto & \left(x_{1}+t,x_{2}\right),\; x_{2}\neq0\label{eq:a3}\\
\left(x_{1},x_{2}\right) & \longmapsto & \left(x_{1},x_{2}+t\right),\; x_{1}\neq0\label{eq:a4}
\end{eqnarray}
lift to unitary one-parameter groups acting on $L^{2}\left(M\right)$;
and it is immediate that the respective infinitesimal generators are
the closed operators $\frac{\partial}{\partial x_{j}}$.

If $\varphi\in C_{c}^{\infty}\left(M\right)$ is supported over some
open set in $\mathbb{R}^{2}\backslash\left\{ \left(0,0\right)\right\} $,
for example, $\left(x_{1}-2\right)^{2}+x_{2}^{2}<1$, if $1<s<2$,
$1<t<2$, then the two functions 
\begin{equation}
U_{1}\left(s\right)U_{2}\left(t\right)\varphi\;\mbox{and }U_{2}\left(t\right)U_{1}\left(s\right)\varphi\label{eq:a5}
\end{equation}
are supported on different sheets in the covering $M\longrightarrow\mathbb{R}^{2}\backslash\left\{ \left(0,0\right)\right\} $,
two levels opposite; see Fig \ref{fig:trans}. Hence the two unitary
groups $\left\{ U_{1}\left(s\right)\right\} _{s\in\mathbb{R}}$ and
$\left\{ U_{2}\left(t\right)\right\} _{t\in\mathbb{R}}$ do not commute. 

Hence it follows from Nelson's theorem \cite{Nel59} that $L$ in
(\ref{eq:a2}) is not essentially selfadjoint. Since $L\leq0$ (in
the sense of Hermitian operators) its deficiency indices are equal.
It was proved in \cite{Tia11} that the indices are $\left(\infty,\infty\right)$;
see also details below.
\end{proof}
\begin{figure}[H]
\includegraphics[scale=0.6]{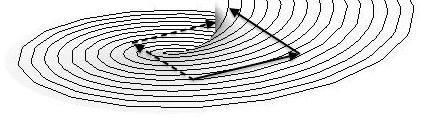}

\protect\caption{\label{fig:trans}Translation of $\varphi$ to different sheets.}
\end{figure}

\begin{prop}
\label{prop:irred}Let $M$ be the $\infty$-covering Riemann surface
of $\log z$, and let 
\[
\left\{ U_{1}\left(s\right)\right\} _{s\in\mathbb{R}},\;\left\{ U_{2}\left(t\right)\right\} _{t\in\mathbb{R}}
\]
be the two unitary one-parameter groups in $L^{2}\left(M\right)$
from Proposition \ref{prop:logz}. Then the two groups together act
irreducibly on $L^{2}\left(M\right)$. \end{prop}
\begin{proof}
We shall need the following lemma. \end{proof}
\begin{lem}
Let $\varepsilon,s,t\in\mathbb{R}_{+}$, and pick a sheet $M_{k}$
of the surface $M$(see Fig \ref{fig:logz}). On $M_{k}$, consider
the ``box'' $B_{s,t}=\left\{ \left(x_{1},x_{2}\right)\:\big|\:\varepsilon<x_{1}<s,\varepsilon<x_{2}<t\right\} $,
then the commutator
\begin{equation}
C\left(s,t\right)=U_{1}\left(s\right)U_{2}\left(t\right)U_{1}\left(-s\right)U_{2}\left(-t\right)\label{eq:a7}
\end{equation}
acts as the identity operator in $B_{s,t}$. \end{lem}
\begin{proof}
This follows from the reasoning below (\ref{eq:a5}) and Fig \ref{fig:trans}.
\end{proof}

\begin{proof}[Proof of Proposition \ref{prop:irred} continued]
If $P$ is a projection in $L^{2}\left(M\right)$ which commutes
with all the operators $C\left(s,t\right)$ in (\ref{eq:a7}), then
$P$ must be a multiplication operator; say multiplication by some
function $f$, $f=\overline{f}=f^{2}$. If $P$ also commutes with
each of the unitary one-parameter groups $\left\{ U_{1}\left(s\right)\right\} $
and $\left\{ U_{2}\left(t\right)\right\} $, it follows from that
(\ref{eq:a3})-(\ref{eq:a4}) that the function $f$ must be constant;
hence $f\equiv0$, or $f\equiv1$. Consequently $P=0$, or $P=I_{L^{2}\left(M\right)}$,
and it follows that the system $\left\{ U_{1}\left(s\right),U_{2}\left(t\right)\right\} $
is irreducible.
\end{proof}

Let $X_{j}=\frac{\partial}{\partial x_{j}}\Big|_{C^{\infty}\left(M\right)}$,
$j=1,2$, as in Proposition \ref{prop:logz}. We see that each $X_{j}$
is essentially skew-adjoint, but $e^{t\overline{X_{1}}}e^{s\overline{X_{2}}}\neq e^{s\overline{X_{2}}}e^{t\overline{X_{1}}}$
globally. Set 
\begin{equation}
\mathscr{D}_{\infty}:=\bigcap_{l_{1},l_{2}\in\mathbb{N}}\mathscr{D}(\overline{X}_{1}^{l_{1}}\overline{X}_{2}^{l_{2}})=\bigcap_{n=1}^{\infty}dom\left(\overline{L}^{n}\right)\label{eq:mdom}
\end{equation}
where $l_{1},l_{2}\in\mathbb{N}$. Let 
\begin{eqnarray}
L & := & \Delta\Big|_{\mathscr{D}_{\infty}}\label{eq:L1}\\
L' & := & \Delta\Big|_{C_{c}^{\infty}\left(M\right)}\label{eq:L2}
\end{eqnarray}
where $\Delta=(\frac{\partial}{\partial x_{1}})^{2}+(\frac{\partial}{\partial x_{2}})^{2}$. 

By Nelson's theorem \cite{Nel59}, we have the following equivalence
\begin{align*}
\mbox{the two operators } & \overline{X_{1}},\overline{X_{2}}\;\mbox{commute strongly}\\
 & \Updownarrow\\
L^{*} & =L
\end{align*}
Since $L\leq0$, it suffices to consider the deficiency space
\begin{equation}
\mathscr{D}_{1}(L):=\{\psi\in\mathscr{D}(L^{*}):L^{*}\psi=\psi\}.\label{eq:D1}
\end{equation}
By elliptic regularity, we have $\psi\in C^{\infty}(M)\cap L^{2}(M)$,
for all $\psi\in\mathscr{D}_{1}\left(L\right)$. 
\begin{rem}
There is a distinction between the two versions of Laplacian. For
example, if $M=\mathbb{R}^{2}\backslash\left\{ \left(0,0\right)\right\} $,
i.e., the punctured plane, then $\mathscr{D}_{\infty}=C^{\infty}$-vectors
for the unitary representation of $G=\left(\mathbb{R}^{2},+\right)$
on $L^{2}\left(M\right)$. In this case, $L$ is the free Hamiltonian
which has deficiency indices $\left(0,0\right)$, however $L'$ has
indices $\left(1,1\right)$. The corresponding unitary groups are
translations in the two coordinate directions of $\mathbb{R}^{2}$,
and they do commute. Therefore, $\overline{X_{1}}$ and $\overline{X_{2}}$
strongly commute, and the two dimension Lie algebra with generators
$X_{j}$ is integrable. 

Since in the $\log z$ example, $\overline{X_{1}}$ and $\overline{X_{2}}$
do \uline{not} strongly commute, it follows that $X_{1}^{2}+X_{2}^{2}$
has deficiency indices $\left(m,m\right)$, $m>0$. We show below
that $m=\infty$ \cite{Tia11}. \end{rem}
\begin{lem}
\label{lem:Kv}Let $K_{\nu}$ be the modified Bessel function of the
second kind of order $\nu$, and suppose $\nu\in(-1,1)$. Then 
\begin{equation}
\int_{0}^{\infty}\left|K_{\nu}(z)\right|^{2}zdz=\frac{1}{2}\frac{\pi\nu}{\sin\pi\nu}.\label{eq:logz inf sur12-1}
\end{equation}
\end{lem}
\begin{proof}
By Nicholson's integral representation of $K_{\nu}$ \cite[pg. 440]{Wat44},
we have 
\begin{equation}
K_{\mu}(z)K_{\nu}(z)=2\int_{0}^{\infty}K_{\mu+\nu}(2z\cosh t)\cosh((\mu-\nu)t)dt.\label{eq:logz inf sur12-2}
\end{equation}
Since $K_{\nu}$ is real-valued, setting $\mu=\nu$, it follows that
\begin{equation}
\left|K_{\nu}(z)\right|^{2}=2\int_{0}^{\infty}K_{2\nu}(2z\cosh t)dt\label{eq:logz inf sur13}
\end{equation}

We recall the following identity \cite[pg. 388, eq. (8)]{Wat44}:
\begin{equation}
\int_{0}^{\infty}K_{\nu}(z)z^{\beta-1}dz=2^{\beta-2}\Gamma\left(\frac{\beta+\nu}{2}\right)\Gamma\left(\frac{\beta-\nu}{2}\right),\ \Re(\beta)>\left|\Re(\nu)\right|.\label{eq:logz inf sur14}
\end{equation}
For $\nu\in(-1,1)$, $\beta=2$, (\ref{eq:logz inf sur14}) yields
\begin{eqnarray}
\int_{0}^{\infty}K_{2\nu}(2z\cosh t)z\: dz & = & \frac{1}{(2\cosh t)^{2}}\Gamma\left(\frac{2+2\nu}{2}\right)\Gamma\left(\frac{2-2\nu}{2}\right)\nonumber \\
 & = & \frac{1}{(2\cosh t)^{2}}\Gamma\left(1+\nu\right)\Gamma\left(1-\nu\right)\nonumber \\
 & = & \frac{1}{(2\cosh t)^{2}}\nu\Gamma\left(\nu\right)\Gamma\left(1-\nu\right)\nonumber \\
 & = & \frac{1}{(2\cosh t)^{2}}\left(\frac{\pi\nu}{\sin\pi\nu}\right);\label{eq:logz inf sur15}
\end{eqnarray}
where the last step follows from the identity 
\begin{equation}
\Gamma\left(\nu\right)\Gamma\left(1-\nu\right)=\frac{\pi}{\sin\pi\nu}\label{eq:logz inf sur16}
\end{equation}
of the Gamma function $\Gamma$.

Combining (\ref{eq:logz inf sur13}) and (\ref{eq:logz inf sur15}),
and using Fubinis's theorem, we get 
\begin{eqnarray*}
\int_{0}^{\infty}\left|K_{\nu}(z)\right|^{2}zdz & = & 2\int_{0}^{\infty}\left(\int_{0}^{\infty}K_{2\nu}(2z\cosh t)zdz\right)dt\\
 & = & 2\left(\frac{\pi\nu}{\sin\pi\nu}\right)\int_{0}^{\infty}\frac{1}{(2\cosh t)^{2}}dt\\
 & = & \frac{1}{2}\frac{\pi\nu}{\sin\pi\nu}\int_{0}^{\infty}\frac{1}{\cosh^{2}t}dt\\
 & = & \frac{1}{2}\frac{\pi\nu}{\sin\pi\nu}.
\end{eqnarray*}

\end{proof}
Let $M$ = the Riemann surface of complex $\log z$. Note that $M$
is covered by a single coordinate patch under polar coordinates, i.e.,
\begin{align*}
x & =r\cos\theta\\
y & =r\sin\theta
\end{align*}
where $r\in\mathbb{R}_{+}$, $\theta\in\mathbb{R}$; and it has the
standard metric 
\[
ds^{2}=dr^{2}+r^{2}d\theta^{2}.
\]
Taking Fourier transform in the $\theta$ variable leads to the decomposition
\begin{eqnarray}
L^{2}(M) & = & \int_{\mathbb{R}}^{\oplus}\mathscr{H}_{\xi}\: d\xi,\;\mbox{where}\label{eq:log1}\\
\mathscr{H}_{\xi} & := & L^{2}(\mathbb{R}_{+},rdr)\otimes span\{e^{i\xi\theta}\}.\label{eq:log2}
\end{eqnarray}
Specifically, for all $f\in L^{2}\left(M\right)$, we set
\begin{eqnarray*}
\widehat{f}\left(r,\xi\right) & := & \int_{-\infty}^{\infty}f\left(r,\theta\right)e^{-i\xi\theta}d\theta;\;\mbox{then}\\
f\left(r,\theta\right) & = & \frac{1}{2\pi}\int_{-\infty}^{\infty}\widehat{f}\left(r,\xi\right)e^{i\theta\xi}d\xi,\;\mbox{and}\\
\left\Vert f\right\Vert _{L^{2}\left(M\right)}^{2} & = & \frac{1}{2\pi}\int_{-\infty}^{\infty}\left(\int_{0}^{\infty}\left|\widehat{f}\left(r,\xi\right)\right|^{2}rdr\right)d\xi
\end{eqnarray*}
The formal $2$ dimensional Laplacian takes the form 
\begin{equation}
\Delta=\int^{\oplus}\left(\frac{1}{r}\frac{d}{dr}\left(r\frac{d}{dr}\right)-\frac{\xi^{2}}{r^{2}}\right)\otimes1.\label{eq:log3}
\end{equation}

\begin{prop}
Let $L$ be the Laplacian in (\ref{eq:L1}), and $\mathscr{D}_{1}\left(L\right)$
the deficiency space (\ref{eq:D1}) as before. Then, $\psi\in\mathscr{D}_{1}\left(L\right)$,
i.e., $\psi$ is a solution to the following equation 
\[
\Delta\psi=\psi,\;\psi\in C^{\infty}\left(M\right)\cap L^{2}\left(M\right)
\]
iff there is a Borel function $g$ supported in $(-1,1)$ and satisfies
\begin{equation}
\int_{-1}^{1}\frac{1}{2}\frac{\pi\xi}{\sin\pi\xi}\left|g\left(\xi\right)\right|^{2}d\xi<\infty\label{eq:log6}
\end{equation}
such that 
\begin{equation}
\psi(r,\theta)=\frac{1}{2\pi}\int_{-1}^{1}g(\xi)K_{\xi}(r)e^{i\xi\theta}d\xi.\label{eq:log7}
\end{equation}
Here, $K_{\xi}$ is the modified Bessel function of the second kind.
Consequently, $L$ has deficiency indices $\left(\infty,\infty\right)$.\end{prop}
\begin{proof}
Let $\psi\in C^{\infty}\left(M\right)\cap L^{2}\left(M\right)$. Using
the decomposition (\ref{eq:log1})-(\ref{eq:log2}), we have 
\[
\psi\left(r,\theta\right)=\frac{1}{2\pi}\int_{-\infty}^{\infty}\left(\int_{0}^{\infty}\widehat{\psi}\left(r,\xi\right)e^{i\xi\theta}rdr\right)d\xi.
\]
It follows from (\ref{eq:log3}) that 
\begin{eqnarray*}
\Delta\psi & = & \frac{1}{2\pi}\int_{-\infty}^{\infty}\left(\int_{0}^{\infty}\Delta\left(\widehat{\psi}\left(r,\xi\right)e^{i\xi\theta}\right)rdr\right)d\xi\\
 & = & \frac{1}{2\pi}\int_{-\infty}^{\infty}\left(\int_{0}^{\infty}\left(\frac{1}{r}\frac{d}{dr}\left(r\frac{d}{dr}\right)-\frac{\xi^{2}}{r^{2}}\right)\widehat{\psi}\left(r,\xi\right)e^{i\xi\theta}rdr\right)d\xi.
\end{eqnarray*}
Hence $\Delta\psi=\psi$ iff 
\begin{equation}
\left(\frac{1}{r}\frac{d}{dr}\left(r\frac{d}{dr}\right)-\frac{\xi^{2}}{r^{2}}\right)\widehat{\psi}\left(r,\xi\right)=\widehat{\psi}\left(r,\xi\right);\label{eq:log4}
\end{equation}
where
\[
l_{\xi}:=\frac{1}{r}\frac{d}{dr}\left(r\frac{d}{dr}\right)-\frac{\xi^{2}}{r^{2}}
\]
is the Bessel differential operator of order $\xi$, acting in $L^{2}\left(\mathbb{R}_{+},rdr\right)$.

Note that $l_{\xi}\big|_{C_{c}^{\infty}\left(\mathbb{R}_{+}\right)}\leq0$
(in the order of Hermitian operators); and it is essentially selfadjoint
in $L^{2}\left(\mathbb{R}_{+},rdr\right)$ iff $\left|\xi\right|\geq1$,
see e.g., \cite{AG93}. Thus in the solution to (\ref{eq:log4}),
we must have $\xi\in\left(-1,1\right)$. The corresponding defect
vector is given by 
\[
\widehat{\psi}\left(r,\xi\right)=\mbox{const}\cdot K_{\xi}(r);
\]
where $K_{\xi}$ denotes the modified Bessel function of the second
kind. Since the constant depends on $\xi$, we may write 
\[
\widehat{\psi}\left(r,\xi\right)=g\left(\xi\right)K_{\xi}(r)
\]
subject to the condition 
\begin{equation}
\int_{-1}^{1}\left(\int_{0}^{\infty}\left|g\left(\xi\right)K_{\xi}(r)\right|^{2}rdr\right)d\xi<\infty.\label{eq:log5}
\end{equation}
By Fubini's theorem, 
\begin{eqnarray*}
\mbox{LHS}_{\left(\ref{eq:log5}\right)} & = & \int_{-1}^{1}\left(\int_{0}^{\infty}\left|K_{\xi}(r)\right|^{2}rdr\right)\left|g\left(\xi\right)\right|^{2}d\xi\\
 & = & \int_{-1}^{1}\frac{1}{2}\frac{\pi\xi}{\sin\pi\xi}\left|g\left(\xi\right)\right|^{2}d\xi,\;\mbox{by Lemma }\ref{lem:Kv};
\end{eqnarray*}
which gives (\ref{eq:log6}).

Finally, we have 
\begin{eqnarray*}
\psi(r,\theta) & = & \frac{1}{2\pi}\int_{-\infty}^{\infty}\widehat{\psi}\left(r,\xi\right)e^{i\xi\theta}d\xi\\
 & = & \frac{1}{2\pi}\int_{-1}^{1}g(\xi)K_{\xi}(r)e^{i\xi\theta}d\xi
\end{eqnarray*}
which is the desired result. 
\end{proof}

\begin{rem}
The harmonic analysis of the Riemann surface $M$ of $\log z$ is
of independent interest, but it will involve von Neumann algebras
and non-commutative geometry. As we noted, to study this, we are faced
with two non-commuting unitary one-parameter groups acting on $L^{2}(M)$
(corresponding to the two coordinates for $M$). Of interest here
is the von Neumann algebra generated by these two non-commuting unitary
one parameter groups. It is likely that this von Neumann algebra is
a type III factor. There is a sequence of interesting papers by K.
Schmudgen on dealing with some of this \cite{MR755571,MR774726,MR808690,MR847352,MR829589}.

In any case, the properties of the von Neumann algebra depend on the
defect space (\ref{eq:D1}) for the $M$-Laplacian $L$ . The role
of the non-commutativity is tied in with the operator $L$ as follows:

Nelson's theorem on analytic vectors \cite{Nel59} applies more generally
to Lie algebras of operators, commutative or not. We summarize briefly
its relevance. Let $\mathfrak{g}$ be a finite-dimensional real Lie
algebra of skew symmetric operators with a common dense domain in
a fixed Hilbert space. Pick a basis for $\mathfrak{g}$, and let $L$
be the sum of squares of the basis-elements; the Nelson-Laplacian.
The first theorem in \cite{Nel59} states that $L$ analytically dominates
the Lie algebra $\mathfrak{g}$. The notion of \textquotedblleft analytic
domination\textquotedblright{} is powerful. It means that analytic
vectors for $L$ are also analytic for the whole Lie algebra $\mathfrak{g}$.
As a Corollary: If $L$ is essentially selfadjoint, it has a dense
space of analytic vectors, and so these will also be analytic for
$\mathfrak{g}$, and so $\mathfrak{g}$ is integrable.

It the commutative case, if $\dim\mathfrak{g}=n$, and if $L$ is
essentially selfadjoint, then it follows that the Lie algebra $\mathfrak{g}$
is integrable; i.e., we have a unitary representation $\mathcal{U}$
of $\mathbb{R}^{n}$ such that $d\mathcal{U}=\mathfrak{g}$; see Definition
\ref{def:int}. Hence the operators in $\mathfrak{g}$ are essentially
skew-adjoint, and they strongly commute. The converse implication
holds as well.

Returning to $L^{2}(M)$: Since in our $\log z$ example (Examples
\ref{ex:1D}-\ref{ex:log}), the two unitary one-parameter groups
do not commute, it follows that $\mathfrak{g}$ is not integrable,
and so the Nelson-Laplacian $L$ must have indices $(m,m)$, $m>0$. \end{rem}
\begin{thm}[\cite{Jor76}]
\label{thm:flow3}Let $\mathscr{H}$ be a Hilbert space. Let $\mathfrak{g}\subset\mathscr{L}\left(\mathscr{H}\right)$
be a finite dimensional Lie algebra. Suppose $\mathfrak{g}$ is generated
by a subset $S$ such that every $A\in S$ is closable and the closure
$\overline{A}$ generates a $C_{0}$ group $\left\{ \pi\left(t,A\right)\right\} _{t\in\mathbb{R}}\subset\mathscr{L}\left(\mathscr{H}\right)$.
Then $\mathfrak{g}$ is integrable iff $\mathfrak{g}$ has a dense
locally invariant $\mathscr{D}$ in $\mathscr{H}$.\end{thm}
\begin{rem}
Theorem \ref{thm:flow3} also applies to abelian Lie algebras. In
Example \ref{ex:log}, we have two derivative operators acting on
$L^{2}\left(M\right)$, where $M$ is the Riemann surface of $\log z$.
The Lie algebra is two dimensional, and not integrable. We conclude
that there is no locally invariant $\mathscr{D}$ for the two operators. 

There is a local representation in $L^{2}(M)$ but not a global one.
By global we mean the closures are strongly commuting, which we do
not have. In our $\log z$ example it is obvious that we can integrate
locally $\varphi_{g}$, $g$ in a small neighborhood of 0 in $\mathbb{R}^{2}$,
so a local representation $\varphi$ of $\mathbb{R}^{2}$ acting on
$L^{2}(M)$, but $\varphi$ will not have any locally invariant $\mathscr{D}$.
This is different from our 1D examples. In Example \ref{ex:1D}, we
had a $(1,1)$ example, but it is contained in a $(0,0)$ example
which has a locally invariant $\mathscr{D}$. No such thing happens
for $L^{2}(M)$ since the two operators are already essentially skew-adjoint.\end{rem}
\begin{cor}
\label{cor:inv}Let $G$ be a simply connected Lie group with Lie
algebra $\mathfrak{g}$, and exponential mapping $\mathfrak{g}\xrightarrow{\exp}G$.
Let $\mathscr{H}$ be a Hilbert space, and $\mathscr{D}\subset\mathscr{H}$
a dense subspace. Let $\rho\in Sk\left(\mathscr{D},\mathscr{H}\right)$
be a representation of $\mathfrak{g}$ with $\mathscr{D}$ as a common
dense domain for the skew-symmetric operators $\left\{ \rho\left(x\right)\:\big|\: x\in\mathfrak{g}\right\} $.
By a \uline{local representation} for $\rho$, we mean a neighborhood
$W$ of $e$ in $G$ and a mapping 
\begin{equation}
\varphi_{W}:W\longrightarrow\mbox{operators on }\mathscr{D}\;(\mbox{generally unbounded})\label{eq:lrep1}
\end{equation}
such that if $g_{1},g_{2}$ and $g_{1}g_{2}$ are in $W$, then the
following two conditions 
\begin{equation}
\varphi_{W}\left(g_{1}g_{2}\right)=\varphi_{W}\left(g_{1}\right)\varphi_{W}\left(g_{2}\right),\;\mbox{and}\label{eq:lrep2}
\end{equation}
\begin{equation}
\frac{d}{dt}\varphi_{W}\left(\exp\left(tx\right)\right)=\rho\left(x\right)\varphi_{W}\left(\exp\left(tx\right)\right)\label{eq:lrep3}
\end{equation}
hold where $x\in\mathfrak{g}$, and $\exp\left(tx\right)\in W$. 

Suppose that there is a system of neighborhoods $\mathcal{W}=\left\{ W\right\} $
of $e$ in $G$, solutions $\varphi_{W}$ to (\ref{eq:lrep1})-(\ref{eq:lrep3}),
and subspaces $\mathscr{D}_{W}\subset\mathscr{D}$ such that the following
two conditions hold: 
\begin{equation}
\mathscr{D}=\bigcup_{W\in\mathcal{W}}\mathscr{D}_{W},\;\mbox{and}\label{eq:lrep4}
\end{equation}
\begin{equation}
\varphi_{W}\left(g\right)\mathscr{D}_{W}\subset\mathscr{D}_{W},\;\forall g\in W.\label{eq:lrep5}
\end{equation}
Then we conclude that $\rho$ is integrable, i.e., there is a unitary
representation $\mathcal{U}$ of $G$, acting on $\mathscr{H}$, such
that 
\[
\mathscr{D}\subseteq\mathscr{H}_{\infty},\;\rho\left(x\right)\subseteq d\mathcal{U}\left(x\right),\;\forall x\in\mathfrak{g};\;\mbox{and}
\]
\[
\left(\mbox{Graph closure of }\rho\left(x\right)\right)=d\mathcal{U}\left(x\right),\;\forall x\in\mathfrak{g}.
\]

\end{cor}

\section*{Appendix. Riemann surfaces of finite-cover degree}

Fix $N\in\mathbb{N}$, and let $M$ be the $N$-covering surface of
$\mathbb{R}^{2}\backslash\left\{ \left(0,0\right)\right\} $. Under
polar coordinates, $M$ is covered in a single coordinate patch as
\begin{align*}
x & =r\cos\theta\\
y & =r\sin\theta
\end{align*}
where $r\in\mathbb{R}_{+}$, and $\theta\in[0,2\pi N)$; and it has
the induced metric
\[
ds^{2}=dr^{2}+r^{2}d\theta^{2}
\]
with volume form 
\[
dV=rdrd\theta.
\]

Using Fourier series in the $\theta$ variable, we have the following
decomposition 
\begin{equation}
L^{2}\left(M\right)=\sum_{k\in\mathbb{Z}}^{\oplus}\left(L^{2}(\mathbb{R}_{+},rdr)\otimes\mbox{span}\{e^{i\theta k/N}\}\right).\label{eq:FL1}
\end{equation}
See, for example, \cite[Chap. 4]{SW71}. Hence, for all $f\in L^{2}\left(M\right)$,
we set 
\begin{eqnarray*}
\widehat{f}_{k}\left(r\right) & = & \frac{1}{2\pi N}\int_{0}^{2\pi N}f\left(r,\theta\right)e^{-i\theta k/N}d\theta;\;\mbox{then}\\
f\left(r,\theta\right) & = & \sum\widehat{f}_{k}\left(r\right)e^{i\theta k/N}
\end{eqnarray*}
and 
\[
\left\Vert f\right\Vert _{L^{2}\left(M\right)}^{2}=\sum_{k=-\infty}^{\infty}\int_{0}^{\infty}\left|\widehat{f}_{k}\left(r\right)\right|^{2}rdr.
\]

The formal 2D Laplacian in polar coordinates takes the form 
\begin{equation}
\Delta=\sum_{k\in\mathbb{Z}}^{\oplus}\left(\frac{1}{r}\frac{d}{dr}\left(r\frac{d}{dr}\right)-\frac{(k/N)^{2}}{r^{2}}\right)\otimes1\label{eq:FL2}
\end{equation}
Set $W:L^{2}(\mathbb{R}_{+},rdr)\rightarrow L^{2}(\mathbb{R}_{+},dr)$
by 
\begin{equation}
Wf(r):=r^{1/2}f(r).\label{eq:W-1}
\end{equation}
$W$ is unitary and it converts (\ref{eq:FL2}) into 
\begin{align}
W\Delta W^{*} & =\sum_{k\in\mathbb{Z}}^{\oplus}\left(l_{k/N}\otimes1\right),\;\mbox{where}\label{eq:Lap2-1}\\
l_{k/N} & :=\frac{d^{2}}{dr^{2}}-\frac{(k/N)^{2}-1/4}{r^{2}}\label{eq:Lap3-1}
\end{align}
(It is understood that $W$ acts on the radial part of the decomposition
(\ref{eq:FL1}).)

Note that $l_{\nu}$ in (\ref{eq:Lap3-1}) is the Bessel differential
operator of order $\nu$ acting on $L^{2}(\mathbb{R}_{+},dr)$, where
$dr$ denotes the Lebesgue measure. It is known that $l_{\nu}\big|_{C_{c}^{\infty}\left(\mathbb{R}_{+}\right)}$
is essentially selfadjoint iff $\left|\nu\right|\geq1$. See, for
example, \cite{AG93}. 
\begin{prop}
Let $M$ be the $N$-covering surface of $\mathbb{R}^{2}\backslash\left\{ \left(0,0\right)\right\} $,
$N<\infty$. Let $L$ be the Nelson-Laplace operator in (\ref{eq:L1}),
and $\mathscr{D}_{1}$ be the deficiency space in (\ref{eq:D1}).
Then $\mathscr{D}_{1}$ is the linear span of the following functions:
\[
K_{k/N}(r)e^{\pm i\theta(k/N)}
\]
where $k=0,\ldots,N-1$; and $K_{\nu}$ denotes the modified Bessel
function of the second kind of order $\nu$. In particular, $L$ has
deficiency indices $\left(2N-1,2N-1\right)$.\end{prop}
\begin{proof}
Let $\psi\in\mathscr{D}_{1}$, i.e., $\psi$ is the solution to the
following equation:
\[
\Delta\psi=\psi,\;\psi\in C^{\infty}\left(M\right)\cap L^{2}\left(M\right);
\]
see (\ref{eq:D1}). Using (\ref{eq:FL1}), we may write 
\begin{eqnarray}
\psi\left(r,\theta\right) & = & \sum_{k\in\mathbb{Z}}\widehat{\psi}_{k/N}\left(r\right)e^{ik\theta/N},\;\mbox{where}\label{eq:LM1-1}\\
\widehat{\psi}_{k/N}\left(r\right) & := & \frac{1}{2\pi N}\int_{0}^{2\pi N}\psi\left(r,\theta\right)e^{-ik\theta/N}d\theta,\label{eq:LM2-1}
\end{eqnarray}
so that 
\begin{eqnarray*}
\Delta\psi & = & \sum_{k\in\mathbb{Z}}\Delta\left(\sum_{k\in\mathbb{Z}}\widehat{\psi}_{k/N}\left(r\right)e^{ik\theta/N}\right)\\
 & = & \sum_{k\in\mathbb{Z}}\left(\left(\frac{1}{r}\frac{d}{dr}\left(r\frac{d}{dr}\right)-\frac{(k/N)^{2}}{r^{2}}\right)\widehat{\psi}_{k/N}\left(r\right)\right)e^{ik\theta/N}.
\end{eqnarray*}
It follows that $\Delta\psi=\psi$ iff 
\begin{equation}
\left(\frac{1}{r}\frac{d}{dr}\left(r\frac{d}{dr}\right)-\frac{(k/N)^{2}}{r^{2}}\right)\widehat{\psi}_{k/N}\left(r\right)=\widehat{\psi}_{k/N}\left(r\right).\label{eq:dev}
\end{equation}
By \cite{AG93} and the discussion above, the only solution to (\ref{eq:dev})
in $L^{2}\left(\mathbb{R}_{+},rdr\right)$ is a scalar multiple of
$K_{k/N}(r)$, for $\left|k/N\right|<1$, i.e., 
\[
\widehat{\psi}_{k/N}\left(r\right)=K_{k/N}(r),\; k=0,\ldots,N-1.
\]
Therefore, by (\ref{eq:LM1-1}), we have 
\[
\psi\left(r,\theta\right)=\sum_{k\in\mathbb{Z}}\widehat{\psi}_{k/N}\left(r\right)e^{ik\theta/N}=\sum_{k\in\mathbb{Z}}K_{k/N}(r)e^{ik\theta/N}
\]
for $k=0,\ldots,N-1$; which is the assertion.\end{proof}
\begin{acknowledgement*}
The co-authors thank the following colleagues for helpful and enlightening
discussions: Professors Daniel Alpay, Sergii Bezuglyi, Paul Muhly,
Myung-Sin Song, Wayne Polyzou, Gestur Olafsson, Keri Kornelson, and
members in the Math Physics seminar at the University of Iowa.
\end{acknowledgement*}
\bibliographystyle{amsalpha}
\bibliography{number9}

\end{document}